\documentclass{article}
\usepackage{graphicx} 
\usepackage{color}
\usepackage[colorlinks]{hyperref}

\usepackage{dsfont}
\usepackage{amsmath}
\usepackage{amsfonts}
\usepackage{latexsym}
\usepackage{amssymb}
\usepackage{mathtools}
\usepackage{cite}
\newtheorem{thm}{Theorem}[section]
 \newtheorem{theorem}[thm]{Theorem}
 
 \newtheorem{lemma}[thm]{Lemma}
 
 \newtheorem{proposition}[thm]{Proposition}

 \newtheorem{definition}[thm]{Definition}

 \newtheorem{remark}[thm]{Remark}

\newcommand{\beq}{\begin{equation}}
\newcommand{\eeq}{\end{equation}}
\newcommand{\beqa}{\begin{eqnarray}}
\newcommand{\eeqa}{\end{eqnarray}}
\newcommand{\beqas}{\begin{eqnarray*}}
\newcommand{\eeqas}{\end{eqnarray*}}
\newcommand{\bi}{\begin{itemize}}
\newcommand{\ei}{\end{itemize}}

\def\iff{\Leftrightarrow}

\def\RR{\mathds{R}}
\def\ones{\mathds{1}}
\def\pCRM{\text{P-CRM}}
\def\CRM{\text{CRM}}
\usepackage{siunitx}
\usepackage{booktabs}

\title{Parallelizing the Circumcentered-Reflection Method}

\date{}

\begin{document}

\maketitle

\begin{center}
\begin{tabular}{ccc}

\begin{tabular}{c}
Pablo Barros\\
School of Applied Mathematics, FGV\\
Praia de Botafogo, Rio de Janeiro, Brazil\\
{\tt pabloacbarros@gmail.com}
\end{tabular}
&
&
\begin{tabular}{c}
Roger Behling\\
Department of Mathematics, UFSC\\
Blumenau, SC, Brazil\\
{\tt rogerbehling@gmail.com}\\
\end{tabular}\\
&&\\
\begin{tabular}{c}
Vincent Guigues\\
School of Applied Mathematics, FGV\\
Praia de Botafogo, Rio de Janeiro, Brazil\\
{\tt vincent.guigues@fgv.br}
\end{tabular}
&
&
\begin{tabular}{c}
Luiz-Rafael Santos\\
Department of Mathematics, UFSC\\
Blumenau, SC, Brazil\\
{\tt l.r.santos@ufsc.br}
\end{tabular}
\end{tabular}
\end{center}

\begin{abstract}
This paper introduces the Parallelized Circumcentered Reflection Method (\pCRM), a circumcentric approach that parallelizes the Circumcentered Reflection Method (CRM) for solving Convex Feasibility Problems in affine settings. Beyond feasibility, \pCRM{} solves the best approximation problem for any finite collection of affine subspaces; that is, it not only finds a feasible point but directly computes the projection of an initial point onto the intersection. Within a fully self-contained scheme, we also introduce the Framework for the Simultaneous Projection Method (F-SPM) which includes Cimmino’s method as a special case.
 Theoretical results show that both \pCRM{} and F-SPM achieve linear convergence. Moreover, \pCRM{} converges at a rate that is at least as fast as, and potentially superior to, the best convergence rate of F-SPM. As a byproduct, this also yields a new and simplified convergence proof for Cimmino’s method. Numerical experiments show that \pCRM{} is competitive compared to CRM and indicate that it offers a scalable and flexible alternative, particularly suited for large-scale problems and modern computing environments.\\  
\end{abstract}

		

\section{Introduction}

 	Let $\RR^n$ denote the standard $n$-dimensional Euclidean space equipped with  inner product and norm denoted by $\left\langle \cdot,\cdot\right\rangle $
	and $\|\cdot\|$, respectively.

  The Convex Feasibility Problem in \( \RR^n \) consists of finding a point in the intersection of a finite collection of closed convex sets \( \{ U_i \}_{i=1}^m \), where each \(U_i \subset \RR^n\). Denoting this intersection by \( S \), and assuming \( S \neq \emptyset \), the problem is formally stated as:
        \begin{align}
        \text{find } s^{\star} \in S := \bigcap_{i=1}^m U_i. \label{CFP} \tag{CFP}        
        \end{align}
        The goal is to identify a point \( s^{\star} \) that simultaneously satisfies all constraints imposed by the sets \( U_i \), each of which contributes to defining the feasible region \( S \).

Projection methods are a class of iterative algorithms designed to solve this problem by leveraging the concept of projecting onto individual sets. These methods aim to achieve convergence to a solution through repeated projections, with the solution emerging as the limit of this iterative process.

The Circumcentered Reflection method (CRM) utilizes the circumcenter
$\text{circ}(x_0, x_1, x_2, \ldots, x_m)$
of a set of points $x_0,x_1,\ldots,x_m$, to accelerate classical projection methods. Introduced in \cite{crm1}, CRM has since inspired a substantial body of research such as \cite{crm_inspired1, crm_inspired2, crm_inspired3, crm_inspired4, crm_inspired5,Araujo:2022, Arefidamghani:2021, Behling:2021a, Behling:2024, Behling:2024b,Behling:2024c,Behling:2023} and the present one.

Indeed, CRM addresses a stronger version of \ref{CFP} for affine sets: specifically, the problem of projecting onto the intersection $S$ given that $U_1, \ldots, U_m$ are affine subspaces, also known as the \textit{Linear Best Approximation Problem}, stated as
\begin{align}
\text{given }x \in \RR^n, \text{ find } \bar{x} \in S \text{ such that } \|\bar{x} - x\| = \min_{s \in S} \|s - x\|. \label{LBAP} \tag{LBAP}
\end{align}

We need the two following definitions.
\begin{definition} Let $C$ be a closed, convex, and nonempty set of $\RR^n$
and $x \in \RR^n$.
The projection $P_C(x)$ of $x$ 
onto $C$ is the unique point $z \in C$
such that 
$$\|z-x\| = \emph{dist}(x,C) := \underset{w \: \in \: C}{\inf} \: \|w-x\|.$$
\end{definition}

\begin{definition} Let $C$ be a closed, convex, and nonempty set of $\RR^n$
and $x \in \RR^n$.
The reflection of $x$ with respect to $C$ is $R_C(x) := 2P_C(x) - x.$
\end{definition}

As presented in \cite{crm2}, CRM for solving \ref{LBAP} works as follows: given the current point \( z_k \in \RR^n \), the next iterate of CRM is defined by
\[
z_{k+1}^{\text{CRM}} = \text{T}_{\text{CRM}}(z_k) := \text{circ}(z_k, R_{U_1}(z_k), R_{U_2} R_{U_1}(z_k), \ldots, R_{U_m} \cdots R_{U_2} R_{U_1}(z_k)),
\]
where \( R_{U_1}, R_{U_2}, \ldots, R_{U_m} : \RR^n \to \RR^n \) are reflections through sets \( U_1, U_2, \ldots, U_m\).\\

In this context, the contributions of this work are as follows.\\

\par {\textbf{Framework for SPM.}} We introduce F-SPM (Framework for the Simultaneous Projection Method), a general scheme for solving \ref{LBAP} that includes the classical Simultaneous Projection Method, also known as Cimmino's method \cite{Cimmino1938}, as a special case. We establish the convergence of F-SPM to a solution of \ref{LBAP}, thereby providing a new and simplified convergence proof for Cimmino's method.\\

\par {\textbf{Parallelized CRM.}} We propose a parallelized variant of CRM, termed \pCRM{} (Parallelized Circumcentered Reflection Method), which replaces sequential reflections with simultaneous ones. The next iterate is defined as
\[
z_{k+1}^{\text{\pCRM}} = \text{C}_{\text{P}}(z_k) := \text{circ}(z_k, R_{U_1}(z_k), R_{U_2}(z_k), \ldots, R_{U_m}(z_k)),
\]
allowing all reflections to be computed independently. This parallelized structure offers computational advantages, particularly in parallel processing environments. We show that \pCRM{} achieves a convergence rate at least as fast as the best rate attainable by F-SPM.\\

\par {\textbf{Numerical experiments.}} We present numerical experiments comparing the performance of CRM, \pCRM, and F-SPM on various convex feasibility problems with affine constraint sets. While \CRM{} remains competitive for several problem instances, \pCRM{} performs comparably or better in others, even on a moderately threaded machine.\\

The outline of the paper is as follows. 
In Section \ref{sec:spm}, we present a framework for the 
Simultaneous Projection Method  and show that it converges linearly to a point 
$s^{\star}$ in
$S$ which is the projection of the initial point $x$ onto $S$. In Section \ref{sec:headings}, we introduce a parallelized variant \pCRM{} of CRM which can be seen as a method parallelizing the CRM method. We show that 
\pCRM{} also converges
linearly to a point 
$s^{\star}$ in
$S$ which is the projection of the initial point $x$ onto $S$, at a rate at least as good as the best convergence rate of F-SPM. Numerical experiments are presented in Section \ref{sec:numsim}.

\section{F-SPM: a framework for SPM}\label{sec:spm}

Let us consider
$m$ affine subspaces
$U_1, U_2, \ldots, U_m$ of $\RR^n$ with nonempty intersection $S$.
The F-SPM method to find a point in $S$
is given below.

	\noindent\rule[0.5ex]{1\columnwidth}{1pt}
	
	F-SPM
	
	\noindent\rule[0.5ex]{1\columnwidth}{1pt}
	\begin{itemize}
		\item [0.] Let $x \in \RR^n$, $U_1, U_2, \ldots, U_m$ affine subspaces, and  $p_0 \geq 0$, $p_1,\ldots, p_m > 0$ with $\displaystyle \sum_{i=0}^m p_i = 1$ be given.
	    \item[1.] Compute $P_{U_1}(x)$, $P_{U_2}(x)$, $\ldots$, $P_{U_m}(x)$.
		\item[2.] Compute 
  $$\text{T}_{\text{F-SPM}}(x) = p_0x + p_1P_{U_1}(x) + p_2P_{U_2}(x) + \ldots + p_mP_{U_m}(x).$$
		\item[3.] Do $x \leftarrow \text{T}_{\text{F-SPM}}(x)$ and
go to Step 1.
	\end{itemize}
	\rule[0.5ex]{1\columnwidth}{1pt}

F-SPM iterations compute the sequence
$
(\text{T}_{\text{F-SPM}}^k(x))_k
$ given by
$$
\text{T}_{\text{F-SPM}}^0(x) = x,\quad
\text{T}_{\text{F-SPM}}^k(x) = \text{T}_{\text{F-SPM}}\left(\text{T}_{\text{F-SPM}}^{k-1}(x) \right),\;k \geq 1.
$$
As an example, we can consider
$p_i=\frac{1}{m+1}$ for all $i$ yielding
$$
\text{T}_{\text{F-SPM}}(x) =
\frac{1}{m+1}\left(x+\sum_{i=1}^m P_{U_i}(x)
\right).
$$
Another special case is Cimmino's method, obtained taking
$p_0=0$ and $p_i=\frac{1}{m}$ for $i=1,\ldots,m,$ yielding
$$
\text{T}_{\text{F-SPM}}(x)=\frac{1}{m}\sum_{i=1}^m P_{U_i}(x).
$$
In Theorem \ref{convspm}, we prove the convergence
of the sequence $(\text{T}_{\text{F-SPM}}^k(x))_k$
generated by F-SPM to a solution
of \ref{CFP} which is the projection
\begin{equation}\label{defstar}
s^\star=P_S(x)
\end{equation}
of  the initial point $x$ to $S$. To prove this theorem, we start with some preliminary results in the next subsection.

\subsection{Preliminary results}

We begin by recalling essential results from convex analysis regarding projections onto closed convex sets. The proposition below guarantees the existence and uniqueness of these projections and characterizes them via a fundamental variational inequality. This condition not only underpins the geometric structure of our algorithms but also serves as a key tool in the convergence analysis of both F-SPM and \pCRM.

\begin{proposition}\label{propproj} For a closed, convex set $C \neq \emptyset$ and any $x \in \RR^n,$ there exists a unique $z \in C$ such that
$$\|z-x\| = \emph{dist}(x,C) := \underset{w \: \in \: C}{\inf} \: \|w-x\|.$$
Such $z$ is called the \textit{projection of $x$ onto $C$} and denoted $P_C(x).$ Moreover, for $z \in C,$
$$z=P_C(x) \text{ if and only if } (x-z)^T(w-z) \le 0 \text{ for all } w \in C.$$
\end{proposition}
\begin{proof} Consider a sequence \((w_n)\) of points in $C$ such that $\|x - w_n\| \to \text{dist}(x, C).$
This implies that \( (w_n) \) is bounded. Therefore, this sequence has a convergent subsequence \(w_{n_k} \to z\). The closure of $C$ gives us $z \in C,$ hence
\[
\text{dist}(x, C) \leq \|x - z\| \le \|x-w_{n_k}\| + \|w_{n_k}-z\| \to \text{dist}(x, C),
\]
and we obtain \( \text{dist}(x, C) = \|x - z\| \).

Because of the convexity of $C$, we have 
$$u \in C \iff u=z+t(w-z), \text{ for some } w \in C \text{ and } t \in [0,1].$$
Hence,
\begin{align*}
    \|z-x\| = \text{dist}(x,C) &\iff \|z-x\| \le \|z-x+t(w-z)\|
    \\ &\iff 0 \le -2t(x-z)^T(w-z)+t^2\|(w-z)\|^2
    \\ &\iff (x-z)^T(w-z) \le t\|(w-z)\|^2 \quad \forall \: t \in (0,1]
    \\ &\iff (x-z)^T(w-z) \le 0 \quad \forall \: w \in C.
\end{align*}
To establish uniqueness, we note that if $\|z_1-x\| = \|z_2-x\| = \text{dist}(x,C)$, then, as showed above, 
\begin{align*}
    (z_1-x)^T(z_1-z_2) &\le 0,
    \\ (x-z_2)^T(z_1-z_2) &\le 0.
\end{align*} 
Summing up, we obtain $(z_1-z_2)^T(z_1-z_2) =  \|z_1-z_2\|^2 \le 0$, then $z_1=z_2.$
\end{proof}
\vspace*{0.5cm}
\par The following proposition specializes the previous projection results to the case of affine subspaces. Unlike general convex sets, affine subspaces exhibit additional symmetries that yield sharper identities and structural properties. Part (i) provides an alternative characterization of the projection via orthogonality conditions, which will be instrumental when analyzing algorithmic steps that involve intersections of affine spaces. Parts (ii) and (iii) highlight geometric invariances -- respectively, a Pythagorean identity and a reflection symmetry -- that simplify several norm estimates later in the analysis. Lastly, part (iv) emphasizes the linear-algebraic nature of both projection and reflection onto affine subspaces, which underpins their compatibility with parallel and compositional strategies in our method.

\begin{proposition}\label{propprojaff} Let $S$ be an affine subspace of $\RR^n$ and $x \in \RR^n.$ Then
\begin{itemize}
    \item[(i)] if $\bar{x} \in S,$ we have $\bar{x} = P_S(x)$  if, and only if, $\langle x-\bar{x}, s - \bar{x} \rangle = 0$ for all $s \in S$;
    \item[(ii)] $\|x - P_S(x)\|^2 = \|x - s\|^2 - \|s - P_S(x)\|^2,$ for all $s \in S$;
    \item[(iii)] $\|x - s\| = \|R_S(x) - s\|,$ for all $s \in S$;
    \item[(iv)] $P_S$ and $R_S$ are affine operators.
\end{itemize}
\end{proposition}

\begin{proof}
\begin{enumerate}
    \item[(i)] Since $S$ is affine and $\bar{x} \in S,$
    $$s \in S \iff 2\bar{x}-s \in S.$$
    From Proposition \ref{propproj}, 
    \begin{align}
        \bar{x} = P_S(x) &\iff (x-\bar{x})^T(s-\bar{x}) \le 0, (x-\bar{x})^T((2\bar{x}-s)-\bar{x}) \le 0,\;\forall s \in S, \nonumber
        \\ &\iff (x-\bar{x})^T(s-\bar{x}) = 0,\, \forall s \in S. \label{affproj}
    \end{align}
    \item[(ii)]  Writing \eqref{affproj} for $s, s' \in S$ and substracting the corresponding equalities, we obtain 
    \begin{equation}
    \langle x - P_S(x), s - s' \rangle = 0 \label{projaff2}
    \end{equation}
    for any $x$ and $s, s' \in S$. Relation \eqref{projaff2}
for $s'=P_S(x)$ gives 
\begin{equation}\label{ortxsp}
\langle x - P_S(x), s - P_S(x) \rangle = 0.
\end{equation}
By Pythagoras theorem we deduce
    $$\|P_S(x)-x\|^2 + \|s - P_S(x)\|^2 = \|P_S(x)-x+s-P_S(x)\|^2 = \|x-s\|^2$$
    for all $s \in S$.
    \item[(iii)] Again by Pythagoras theorem and relation
    \eqref{ortxsp}, we have
    $$\|x-s\|^2=\|P_S(x)-x\|^2+\|P_S(x)-s\|^2 = \|P_S(x)-x+P_S(x)-s\|^2 = \|R_S(x)-s\|^2.$$
    \item[(iv)] Given $\alpha \in \RR$, $x,y \in \RR^n$,    let $z = \alpha x + (1-\alpha) y$ and $\bar{z} = \alpha P_S(x) + (1- \alpha) P_S(y).$
    
    Then, for any $s, s' \in S$ we have from \eqref{projaff2}
    \begin{align*}
        0=\alpha \langle x - P_S(x), s - s' \rangle + (1-\alpha) \langle y - P_S(y), s - s' \rangle = \langle z - \bar{z}, s - s' \rangle.
    \end{align*}
    Taking $s' = \bar{z} \in S$ and using (i), we obtain $\bar{z} = P_S(z)$, which shows that $P_S$ is affine.
    
    Since $R_S=2P_C-\text{Id}$ is an affine combination of two affine operators, it is also affine.
\end{enumerate}
\end{proof}

The next lemma explores how projections and reflections interact with intersections of affine subspaces. In particular, it shows that projecting a point onto an affine subspace \( U \), or reflecting it across \( U \), does not alter its projection onto the intersection \( U \cap V \). This invariance plays a key role in justifying the use of intermediate projections or reflections within iterative methods like F-SPM and \pCRM. By preserving the projection onto the intersection, such operations allow for decomposition strategies that remain consistent with the underlying geometric goal: convergence to a point in the intersection.

\begin{lemma}\label{lemafflem1} Let $U$ and $V$ be affine subspaces of $\RR^n$ with $U \cap V \neq \emptyset.$ For any $x \in \RR^n$, it holds that 
$$P_{U \cap V}(P_U(x)) = P_{U \cap V}(R_U(x)) = P_{U \cap V}(x).$$
\end{lemma}

\begin{proof} Applying Proposition \ref{propprojaff} (i), we obtain that for any $s \in U \cap V,$
\begin{equation}\label{firstl}
    \langle P_{U \cap V}(P_U(x))-P_U(x), x - P_U(x) \rangle = 0,
\end{equation}
where we have used the fact that 
$P_{U \cap V}(P_U(x)) \in U$. For every
$s \in U \cap V,$ we have $s \in U$, which implies
using again Proposition \ref{propprojaff} (i) that
\begin{equation}\label{secondl}
\langle s - P_U(x), x - P_U(x) \rangle = 0.
\end{equation}
Combining \eqref{firstl} and \eqref{secondl}, we have for every $s \in U \cap V$ that
\begin{equation}\label{thirdl}
\langle x - P_U(x), s - P_{U \cap V}(P_U(x)) \rangle = 0.
\end{equation}
Applying once again Proposition \ref{propprojaff} (i), we also have
\begin{equation}\label{fourthl}
\langle P_U(x) - P_{U \cap V}(P_U(x)), s - P_{U \cap V}(P_U(x)) \rangle=0
\end{equation}
for every $s \in U \cap V$.
Finally, summing \eqref{thirdl} and \eqref{fourthl}
we get 
$$
\langle x - P_{U \cap V}(P_U(x)), s - P_{U \cap V}(P_U(x)) \rangle = 0,
$$
for every $s \in U\cap V$ which gives, using once again Proposition \ref{propprojaff}, $P_{U \cap V}(P_U(x)) = P_{U \cap V}(x).$

By Proposition \ref{propprojaff}-(iv), we next have 
$$P_{U \cap V}(R_U(x)) = P_{U \cap V}(2P_U(x)-x) = 2P_{U \cap V}(P_U(x))-P_{U \cap V}(x) = P_{U \cap V}(x)$$
and the result follows.
\end{proof}

To study the geometric relationship between affine subspaces, we begin by isolating their linear structure. This is done by associating to each affine subspace a direction subspace, defined via translation by a point in the intersection. These direction subspaces will serve as the foundational objects in defining the angle between subspaces.

\begin{definition}\label{diraff}
 Let $U$ and $V$ be affine subspaces of $\RR^n$ with $U \cap V \neq \emptyset.$
We define the directions 
 \(\hat{U}\) and \(\hat{V}\)
 of $U$ and $V$ by
 subspaces given by \(U - \hat{z}\) and \(V - \hat{z}\), respectively, for an arbitrary but fixed \(\hat{z} \in U \cap V\).
\end{definition}

With the direction subspaces in place, we now introduce the Friedrichs angle, a classical tool to quantify the relative orientation of two subspaces. Its cosine captures the degree of transversality between \(U\) and \(V\), and plays a crucial role in establishing linear convergence for projection-type methods.

\begin{definition}
 Let $U$ and $V$ be affine subspaces of $\RR^n$ with $U \cap V \neq \emptyset.$
The cosine of the Friedrichs angle $\theta_F \in (0, \frac{\pi}{2}]$ between $U$ and $V$ is defined as
    \[c_F(U, V) := \sup \{\langle u, v \rangle \mid u \in \hat{U} \cap (\hat{U} \cap \hat{V})^\perp, \quad v \in \hat{V} \cap (\hat{U} \cap \hat{V})^\perp, \quad \|u\| \leq 1, \quad \|v\| \leq 1\},\]
where \(\hat{U}\) and \(\hat{V}\) are directions of $U$ and
$V$ given in Definition \ref{diraff}.
\end{definition}

The following proposition states that the cosine of the Friedrichs angle between two affine subspaces with nontrivial intersection is always strictly less than one. This non-degeneracy condition ensures that the angle is well-defined and bounded away from zero, which is essential for guaranteeing geometric decay in the iterates.

\begin{proposition} Let $U$ and $V$ be affine subspaces of $\RR^n$ with $U \cap V \neq \emptyset.$ The cosine of the Friedrichs angle between $U$ and $V$ satisfies 
$$0 \leq c_F(U, V) < 1.$$
\end{proposition}

\begin{proof} Notice that 
$$W :=\{ u \in \hat{U} \cap (\hat{U} \cap \hat{V})^\perp; \: \|u\| \le 1\} \times \{v \in \hat{V} \cap (\hat{U} \cap \hat{V})^\perp; \: \|v\| \le 1\}$$
is compact in $\RR^n \times \RR^n.$ As a result, $\sup \{\langle u, v \rangle \mid (u,v) \in W\}$ is attained which implies $c_F(U, V) = \langle u, v \rangle$ for some $(u,v) \in W.$ 

It is evident that $c_F(U,V) \ge \langle 0, 0 \rangle = 0$ and $c_F(U, V) = \langle u, v \rangle \le \|u\|\|v\| \le 1.$

Now if $c_F(U, V) = 1$, we get 
$$c_F(U, V) =\langle u, v \rangle = \|u\|\|v\|=1 \implies u = \alpha v$$ 
for some $\alpha \ge 0$, and then 
$$u \in \left(\hat{U} \cap (\hat{U} \cap \hat{V})^\perp \right) \cap \left(\hat{V} \cap (\hat{U} \cap \hat{V})^\perp \right) = \{0\},$$ 
thus $\langle u, v \rangle=0,$ a contradiction with $\langle u, v \rangle=1$. Therefore, $c_F(U, V) < 1.$ 
\end{proof}
\vspace*{0.5cm}

We now state a key inequality known as the error bound condition or linear regularity for two affine subspaces. This result relates the distance to the intersection \( U \cap V \) to the individual distances to \( U \) and \( V \). The constant \( r(U,V) \), which depends explicitly on the cosine of the Friedrichs angle, governs how well-separated the subspaces are. This inequality is fundamental for our convergence analysis: it ensures that approximate feasibility with respect to each individual set implies approximate feasibility with respect to their intersection, thus enabling global convergence results from local stepwise progress.

\begin{proposition}\label{proptwosp} Let $U$ and $V$ be affine subspaces of $\RR^n$ with $U \cap V \neq \emptyset.$  There exists 
$$
r(U,V)=\sqrt{1+ \frac{4}{1-c_F(U,V)^2}}  >1
$$ such that 
$$\emph{dist}(x,U \cap V) \le r(U,V) \max \{\emph{dist}(x, U), \emph{dist}(x, V)\}$$ 
for all $x.$ 
\end{proposition}

\begin{proof} Denote $u = P_U(x), \: v = P_V(x)$, and $s = P_{U \cap V}(x)$. 

By Lemma \ref{lemafflem1}, $s=P_{U \cap V}(u) = P_{U \cap V}(v)$, and by Proposition \ref{propprojaff}-(i), we have
$$s-u \in \hat{U} \cap (\hat{U} \cap \hat{V})^\perp, \: s-v \in \hat{V} \cap (\hat{V} \cap \hat{U})^\perp,$$
and as a result $\langle s-u, s-v \rangle \le c_F(U,V) \|s-u\|\|s-v\|.$

Next, observe that 
\begin{align*}
    (\|x-u\| + \|x-v\|)^2 &\ge \|(s-u)-(s-v)\|^2
    \\ & = \|s-u\|^2+\|s-v\|^2-2 \langle s-u, s-v \rangle
    \\ & \ge \|s-u\|^2 + \|s-v\|^2 - 2 c_F(U,V) \|s-u\| \|s-v\|
    \\ & = (1-c_F(U,V)^2)\|s-u\|^2 + \left( c_F(U,V)\|s-u\|-\|s-v\| \right)^2
    \\ & \ge (1-c_F(U,V)^2)\|s-u\|^2,
\end{align*}
which yields, for $\displaystyle \alpha^2 = \frac{1}{1-c_F(U,V)^2}$,
\begin{equation}\label{subound}
    \|s-u\|^2 \le \alpha^2 (\|x-u\| + \|x-v\|)^2.
\end{equation}
Again by Proposition \ref{propprojaff}-(i), using the fact that $s \in U$, we have
\begin{equation}\label{eqsusu}
\langle x-u,s-u \rangle=0.
\end{equation}
Now we can bound $\text{dist}(x,U \cap V)$ as follows:
\begin{align*}
    \|s - x\|^2 &\stackrel{\eqref{eqsusu}}{=} \|s-u\|^2 + \|x-u\|^2
    \\ & \stackrel{\eqref{subound}}{\le} \alpha^2 (\|x-u\| + \|x-v\|)^2+\|x-u\|^2
    \\ & \le (4\alpha^2+1) \max \{\|x-u\|^2, \|x-v\|^2\}.
\end{align*}
This implies 
\[\|s - x\| \le \sqrt{4\alpha^2+1} \max \{\|x-u\|, \|x-v\|\},\]
which is the desired result.
\end{proof}

We now extend the previous error bound condition for $m$ sets.

\begin{proposition}\label{inducdist}
There exists $r_m>1$ such that 
$$
\displaystyle \|x-s^\star\| \le r_m \max \left\{\emph{dist}(x, U_i),i=1,\ldots,m \right\}
$$ 
for all $x.$    
\end{proposition}

\begin{proof} We prove the result by induction on $m \ge 2$. 
For $m=2$, this is Proposition \ref{proptwosp}. Assuming the result holds
for $k$ sets, we apply Proposition \ref{proptwosp} for $U= \displaystyle \bigcap_{i=1}^{k} U_i$ and $V=U_{k+1}$ to get
\begin{align*}
\|x-s^\star\| = \text{dist}(x,S) &\le r' \max \{\text{dist}(x, U), \text{dist}(x, V)\}
\\ & \le r' \max \left \{r_k \max \left\{\text{dist}(x, U_i),i \le k \right\}, \text{dist}(x, U_{k+1}) \right\}
\\ & \le r'r_k \max \left\{ \text{dist}(x, U_i),i \le k+1\right\},
\end{align*}
for some $r'>1$, $r_k>1$,
which completes the induction step. 
\end{proof}

\vspace*{0.5cm}

In what follows, we denote by  
\[
W_x := \text{aff}\{ x, R_{U_1}(x), R_{U_2}(x), \ldots, R_{U_m}(x) \}
\]  
the affine subspace generated by the current iterate \( x \) and its reflections across the sets \( U_1, \ldots, U_m \). This space plays a central role in the structure of circumcenter-based methods, as it contains all the candidate points used in computing the next iterate. The next lemma shows that projecting any point from \( W_x \) onto the intersection set \( S \) yields the same point \( s^\star \), highlighting a key invariance that underlies the well-posedness and geometric consistency of the circumcenter construction.

\begin{lemma}\label{lemwxb} For any $w \in W_x$, we have
$$P_S(w) = s^\star.
$$
\end{lemma}
\begin{proof}
From Lemma \ref{lemafflem1}, $P_S(x^{(i)}) = P_S(R_{U_i}(x)) = s^\star$ for each $i=1,\ldots,m$. By definition, $w \in W_x$ can be written as $\displaystyle w=\sum_{i=0}^m \alpha_i x^{(i)}$ for real numbers $\alpha_i$ with $\displaystyle \sum_{i=0}^m \alpha_i = 1$. Now, using Proposition \ref{propprojaff}-(iv),
    we have $$P_S(w)=P_S\left(\sum_{i=0}^m \alpha_i x^{(i)}\right)=\sum_{i=0}^m \alpha_i P_S(x^{(i)}) = \sum_{i=0}^m \alpha_i s^\star = s^\star$$
    which achieves the proof of (i).
\end{proof}

The following lemma establishes two fundamental properties of the F-SPM operator. The first ensures that each iterate generated by F-SPM remains within the affine subspace \( W_x \), which is spanned by the initial point and its reflections. The second asserts that all iterates project to the same point \( s^\star \in S \), which reveals a key stability property: the iterates may vary, but their projection onto the intersection set remains invariant. This behavior underpins the convergence analysis and will be essential in proving eventual convergence to \( s^\star \).

\begin{lemma}\label{proptsm} The following assertions hold for F-SPM:
\begin{enumerate}
    \item[(i)] $\text{T}_{\text{F-SPM}}(x) \in W_x$.
    \item[(ii)] $P_S(\text{T}_{\text{F-SPM}}^k(x))_k = s^\star.$
\end{enumerate}
\end{lemma}

\begin{proof}
\begin{enumerate}
    \item[(i)] Because $P_{U_i}(x) = \frac12 x + \frac12 x^{(i)}$ for each $i \ge 1$, it follows that 
    \begin{align*}
        \text{T}_{\text{F-SPM}}(x) &= \sum_{i=0}^m p_iP_{U_i}(x) 
        \\ &= p_0x + \sum_{i=1}^m \left(\frac12 p_i x + \frac12 p_ix^{(i)} \right)
        \\ &= \left(p_0 + \sum_{i=1}^m \frac12 p_i \right)x + \sum_{i=1}^m \frac12 p_ix^{(i)}.
    \end{align*}
    It then suffices to observe that
    $$\left(p_0 + \sum_{i=1}^m \frac12 p_i \right) + \sum_{i=1}^m \frac12 p_i =\sum_{i=0}^m p_i = 1.$$
    \item[(ii)] We prove the result by induction on $k$. The case $k=0$ is obvious. Assuming the result holds for $k-1$, in view of Lemma \ref{lemwxb} and item (i), 
    $$P_S \left(\text{T}_{\text{F-SPM}}^k(x) \right) = P_S \left(\text{T}_{\text{F-SPM}}^{k-1}(x) \right)=s^\star$$
    for $k \ge 1$, completing the induction. 
\end{enumerate}
\end{proof}

\subsection{Convergence of F-SPM}

We now turn to the main convergence result for F-SPM. The theorem below establishes that the sequence generated by the method converges linearly to the projection of the initial point onto the intersection set \( S \). The linear rate depends on the geometric configuration of the sets through the Friedrichs angle and the error bound condition established earlier.

\begin{theorem}\label{convspm} We have $\|\text{T}_{\text{F-SPM}}(x)- s^\star\| \le \displaystyle r_{\text{F-SPM}}\|x-s^\star\|$ for some $r_{\text{F-SPM}} < 1$ and therefore
$$
\|\text{T}_{\text{F-SPM}}^k(x)- s^\star\| \le \displaystyle r_{\text{F-SPM}}^k\|x-s^\star\|
$$
which shows that the sequence
$(\text{T}_{\text{F-SPM}}^k(x))_k$ generated by F-SPM converges
linearly to $s^\star$, the projection of the initial point $x$ onto $S$.
\end{theorem}

\begin{proof} For any $i$, since $s^{\star} \in U_i$, Proposition \ref{propprojaff}-(ii) gives
\begin{align}
    \|P_{U_i}(x)-s^\star\|^2+\|x-P_{U_i}(x)\|^2 = \|x-s^\star\|^2 \label{pyth}
\end{align}
which yields 
\begin{equation}
    \|P_{U_i}(x)-s^\star\| \le \|x-s^\star\|. \label{five}
\end{equation}

By Proposition \ref{inducdist}, for some $1 \le j \le m$ it holds that 
$$\|x-s^\star\| \le r_m \|x-P_{U_j}(x)\|$$
for some $r_m>1$. Note that since $j$ is between 1 and $m$, we have by the
assumption that $p_i >0$ for $i=1,\ldots,m$, that
\begin{equation}\label{pjpos}
p_j >0.
\end{equation}
Together with \eqref{pyth} written for $i=j$, this gives
\begin{align*}
    \|P_{U_j}(x)-s^\star\|^2+\frac{1}{r_m^2}\|x-s^\star\|^2 &\le
    \|P_{U_j}(x)-s^\star\|^2+ 
    \|x-P_{U_j}(x)\|^2 \\
& = \|x-s^\star\|^2 \nonumber \\
\end{align*}

which can be rewritten
\begin{equation}
 \|P_{U_j}(x)-s^\star\| \le \sqrt{1-\frac{1}{r_m^2}} \|x-s^\star\|. \label{six}
\end{equation}

Finally,
\begin{align*}
    \|\text{T}_{\text{F-SPM}}(x)- s^\star\| &= \left\| \left(p_0x+\sum_{i=1}^m p_iP_{U_i}(x)\right) - \sum_{i=0}^m p_i s^\star \right\| 
    \\ & = \left\|p_0(x-s^\star)+ \sum_{i=1}^m p_i(P_{U_i}(x)-s^\star) \right\|
    \\ & \le p_0\|x-s^\star\|+ \sum_{i=1}^m p_i\|P_{U_i}(x)-s^\star\|
    \\ & \le p_0\|x-s^\star\|+ p_j \sqrt{1-\frac{1}{r_m^2}} \|x-s^\star\|+ \sum_{i=1, i \neq j}^m p_i\|x-s^\star\| \tag{by \eqref{five} and \eqref{six}}
    \\ & = \left(1-p_j + p_j\sqrt{1-\frac{1}{r_m^2}} \right) \|x-s^\star\|.
\end{align*}
Therefore the theorem holds with 
\begin{equation}\label{finalspm}
\displaystyle r_{\text{F-SPM}} = 1-p_j + p_j\sqrt{1-\frac{1}{r_m^2}} < 1
\end{equation}
where the strict inequality in \eqref{finalspm}
comes from \eqref{pjpos}.
\end{proof}

In the special case of Cimmino's method, the optimal convergence rate $r_{\text{F-SPM}}^{\star}$ was investigated in \cite{spm_bound} and found to be
\begin{align}\label{bestcimm}
r_{\text{F-SPM}}^{\star}=\cos_p (C_p, D_p)^2 
\end{align}
where we set $U_0 = \RR^n$ and
\[\cos_p (C_p, D_p) = \sup \left\{ \left\| \sum_{j=0}^{m} p_j u_j \right\| : u_j \in \hat{U}_j \cap \left( \bigcap_{i=0}^m \hat{U}_i \right)^\perp, \, j \in \{0,1,\ldots, m\}, \, \sum_{j=0}^{m} p_j \| u_j \|^2 = 1 \right\}\]
is the cosine of the Friedrichs angle between $\displaystyle C_p:=\prod_{j=0}^{m} \hat{U}_j$ and $D_p:= \left\{ (x)_{i=0}^m : x \in \RR^n\right\}.$

\section{Convergence of the Parallelized Circumcentered Reflection Method method}
\label{sec:headings}

\subsection{\pCRM{} method}

Let us consider again
$m$ affine subspaces
$U_1, U_2, \ldots, U_m$ of $\RR^n$ with nonempty intersection $S$.
The \pCRM{} method to find a point in $S$
is given below.

	\noindent\rule[0.5ex]{1\columnwidth}{1pt}
	
	\pCRM
	
	\noindent\rule[0.5ex]{1\columnwidth}{1pt}
	\begin{itemize}
		\item [0.] Let $x \in \RR^n$ 
  and $U_1, U_2, \ldots, U_m$ affine subspaces
  be given.
	    \item[1.] Compute $R_{U_1}(x)$, $R_{U_2}(x)$, $\ldots$, $R_{U_m}(x)$.
		\item[2.] Compute the point $\text{C}_{\text{P}}(x) \in \RR^n$ with the following two properties:
\begin{enumerate}
    \item[(P1)] $\text{C}_{\text{P}}(x) \in W_x := \text{aff}\{ x, R_{U_1}(x), R_{U_2}(x), \ldots, R_{U_m}(x) \}$;
    \item[(P2)] $\text{C}_{\text{P}}(x)$ is equidistant to $x, R_{U_1}(x), R_{U_2}(x), \ldots, R_{U_m}(x)$.
\end{enumerate}
  		\item[3.] Do $x \leftarrow \text{C}_{\text{P}}(x)$ and
go to Step 1.
	\end{itemize}
	\rule[0.5ex]{1\columnwidth}{1pt}

\pCRM{} iterations compute the sequence
$
(\text{C}_{\text{P}}^k(x))_k
$ given by
$$
\text{C}_{\text{P}}^0(x) = x,\quad
\text{C}_{\text{P}}^k(x) = \text{C}_{\text{P}}\left(\text{C}_{\text{P}}^{k-1}(x) \right),\;k \geq 1.
$$
At each iteration, we need to compute the circumcenter $\text{C}_{\text{P}}(x)$
of points $x$, $R_{U_1}(x)$, $R_{U_2}(x)$, $\ldots$, $R_{U_m}(x)$  satisfying (P1) and (P2). 
Each iteration of \pCRM{} requires
therefore more computational effort than F-SPM but we expect less iterations to obtain a solution or an approximate solution of \ref{CFP}  with \pCRM.

We illustrate an iteration of the \pCRM{} method in $\RR^2$ in Figure \ref{fig:pcrmr2}
and in $\RR^3$ in Figure \ref{fig:pcrmr3}.
The figures correspond to the cases where
$U_1$ and $U_2$ are two hyperplanes. The initial point
is $x$ (in $\RR^2$ in Figure \ref{fig:pcrmr2}
and in $\RR^3$ in Figure \ref{fig:pcrmr3})
and we represent on the figures the two affine spaces $A$ and $B$ and five points:
\begin{itemize}
\item an arbitrary initial point $x$;
\item the reflection $R_{U_1}(x)$ of $x$ with respect to $U_1$, used in the computation of the next iterate
$C_P(x)$ of \pCRM;
\item the reflection $R_{U_2}(x)$ of $x$ with respect to $U_2$, used in the computation of the next iterate
$C_P(x)$ of \pCRM;
\item the next iterate $C_P(x)$ (a point in the intersection of $U_1$ and $U_2$).
\end{itemize}

As illustrated in Figures \ref{fig:pcrmr2} and \ref{fig:pcrmr3}, if \( U_1, \ldots, U_m \) are hyperplanes and the initial point \( x \) does not belong to any \( U_i \), then \pCRM{} converges in a single iteration to a point in the intersection \( S \). Indeed, in this setting, each \( U_i \) can be characterized as the set of points equidistant from \( x \) and its reflection \( R_{U_i}(x) \). The reflection satisfies
\[
R_{U_i}(x) = x + 2(P_{U_i}(x) - x),
\]
where \( P_{U_i}(x) \) denotes the orthogonal projection onto \( U_i \). A straightforward computation shows that for any \( z \in \RR^n \),
\[
\|z - R_{U_i}(x)\|^2 = \|z - x\|^2 + 4\langle P_{U_i}(x) - x, z - P_{U_i}(x) \rangle.
\]
Thus, the condition \( \|z - x\| = \|z - R_{U_i}(x)\| \) holds if and only if
\[
\langle P_{U_i}(x) - x, z - P_{U_i}(x) \rangle = 0.
\]
This scalar product relation expresses that the vector \( z - P_{U_i}(x) \) is orthogonal to the normal direction \( P_{U_i}(x) - x \), implying that \( z \) lies on the hyperplane orthogonal to \( P_{U_i}(x) - x \) and passing through \( P_{U_i}(x) \), which defines \( U_i \).

As a result, the intersection \( S = \bigcap_{i=1}^m U_i \) consists of the points simultaneously equidistant from \( x \) and all reflections \( R_{U_i}(x) \). Therefore, the circumcenter \( \text{C}_{\text{P}}(x) \) lies in \( S \), and \pCRM{} achieves convergence in a single iteration.

\begin{figure}
    \centering
    \includegraphics[width=0.7\textwidth]{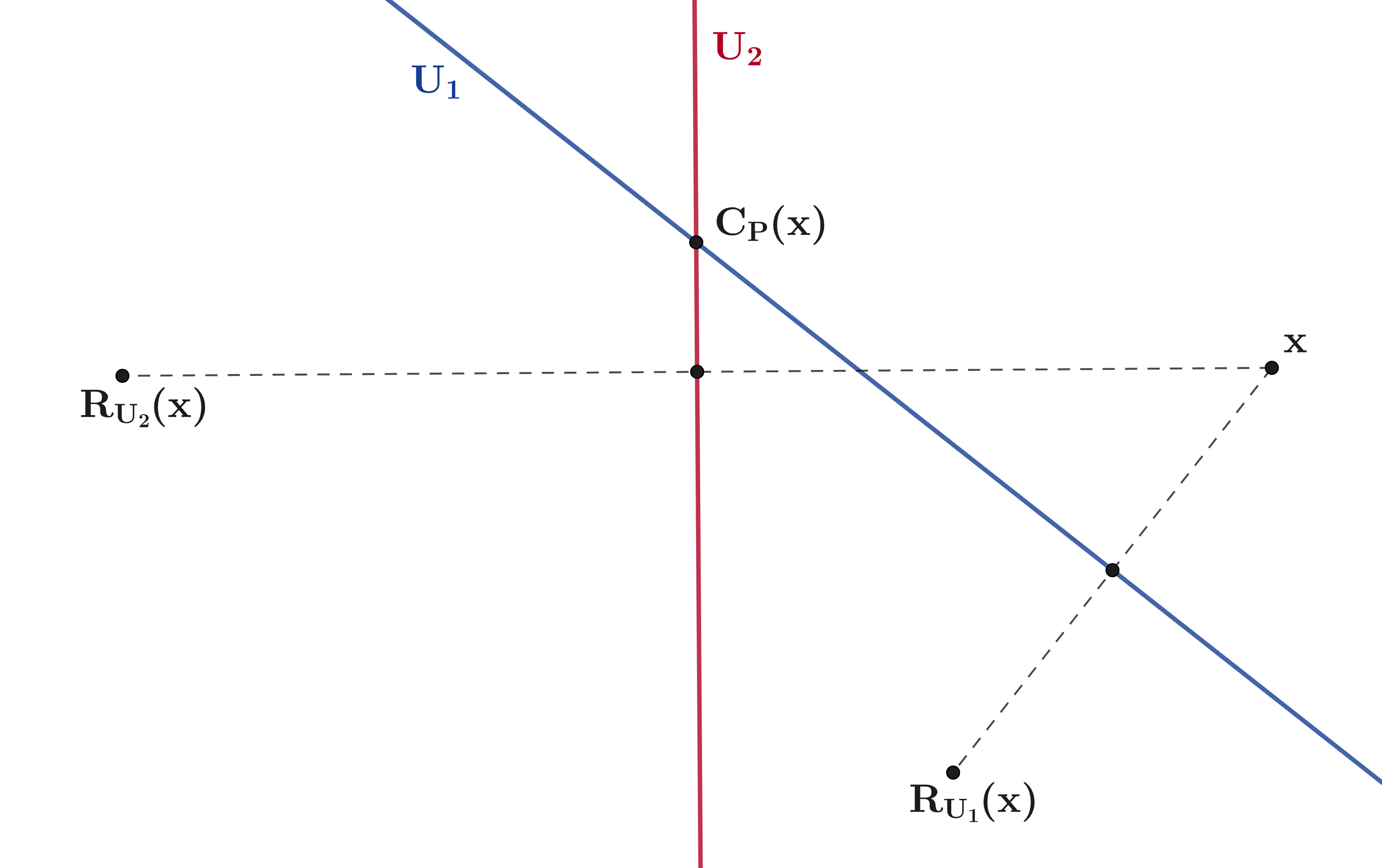}
    \caption{\pCRM{} in $\RR^2$}
    \label{fig:pcrmr2}
\end{figure}

\begin{figure}
    \centering
    \includegraphics[width=0.7\textwidth]{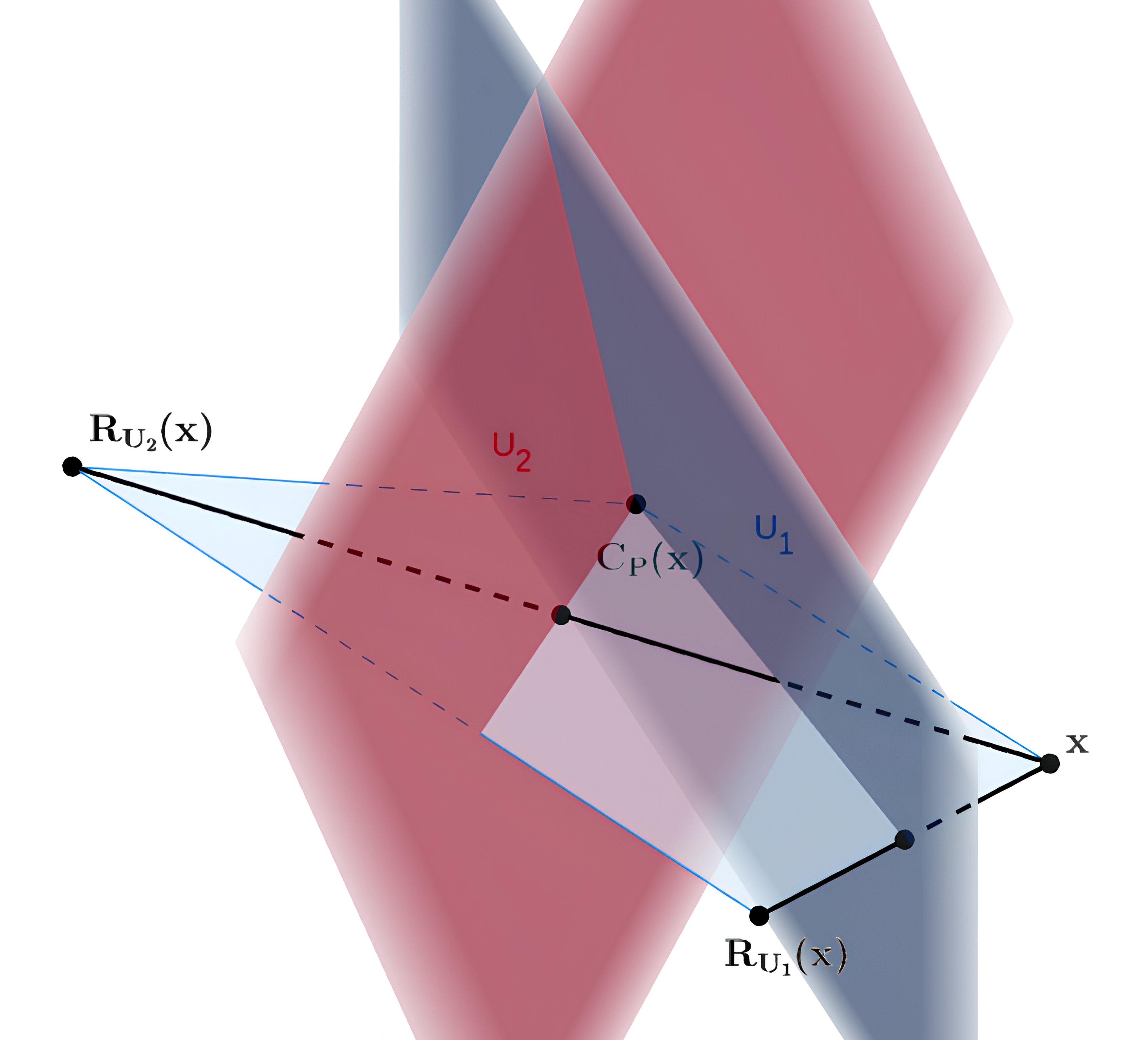}
    \caption{\pCRM{} in $\RR^3$}
    \label{fig:pcrmr3}
\end{figure}

We now explain how to compute the circumcenter of a set of points.
Given \( m+1 \) points \( x_0, x_1, x_2, \ldots, x_m \in \RR^n \), their circumcenter, denoted by \( \text{circ}(x_0, x_1, x_2, \ldots, x_m) \), is defined as the unique point that lies within the affine subspace spanned by these points and that is equidistant from each of them.
By this definition, we must have, for some $\alpha \in \RR^m,$
\[\text{circ}(x_0, x_1, x_2, \ldots, x_m) = x_0 + \sum_{j=1}^m \alpha_j (x_j - x_0),\]
and, for all $i,$
\[
\|\text{circ}(x_0, x_1, x_2, \ldots, x_m)-x_i\| = \|\text{circ}(x_0, x_1, x_2, \ldots, x_m)-x_0\|.
\]
These conditions lead to an \( m \times m \) linear system in the coefficients \( \alpha \in \RR^m \), where the \( i \)-th equation is
\[
\sum_{j=1}^m \alpha_j \langle x_j - x_0, x_i - x_0 \rangle = \frac{1}{2} \| x_i - x_0 \|^2,\]
or
\[
\begin{pmatrix}
\langle x_1 - x_0, x_1 - x_0 \rangle & \langle x_2 - x_0, x_1 - x_0 \rangle & \cdots & \langle x_m - x_0, x_1 - x_0 \rangle \\
\langle x_1 - x_0, x_2 - x_0 \rangle & \langle x_2 - x_0, x_2 - x_0\rangle & \cdots & \langle x_m - x_0, x_2 - x_0 \rangle \\
\vdots & \vdots & \ddots & \vdots \\
\langle x_1 - x_0, x_m - x_0 \rangle & \langle x_2 - x_0, x_m - x_0 \rangle & \cdots & \langle x_m - x_0, x_m - x_0 \rangle
\end{pmatrix}
\begin{pmatrix}
\alpha_1 \\
\alpha_2 \\
\vdots \\
\alpha_m
\end{pmatrix}
= \frac{1}{2}
\begin{pmatrix}
\| x_1 - x_0 \|^2 \\
\| x_2 - x_0 \|^2 \\
\vdots \\
\| x_m - x_0 \|^2
\end{pmatrix}.
\]
Solving this system determines the circumcenter \(\text{circ}(x_0, x_1, x_2, \ldots, x_m)\). However, the uniqueness of \(\alpha\) relies on the linear independence of the vectors \( x_i - x_0 \) (or, equivalently, on the assumption that $x_0,x_1,\ldots,x_m$ are affinely independent), which obviously is not always true.

In Theorem \ref{cvcspmt}, we prove the convergence
of the sequence $(\text{C}_{\text{P}}^k(x))_k$
generated by \pCRM{} to a solution
of \ref{LBAP}, which is the projection
\begin{equation}\label{defstar}
s^\star=P_S(x)
\end{equation}
of the initial point $x$ onto $S$. To prove this theorem, we start with some preliminary results in the next subsection.

\subsection{Preliminary results}

We begin by establishing the existence and uniqueness of the parallel circumcenter, identifying \( \text{C}_{\text{P}}(x) \) as the projection of any point \( s \in S \) onto the affine subspace \( W_x \). This demonstrates that \( \text{C}_{\text{P}}(x) \) is the closest point to \( S \) within \( W_x \). Notably, \( \text{C}_{\text{P}}(x) \) is at least as close to \( S \) as the F-SPM point \( T_{\text{F-SPM}}(x) \).

For short, we use the notation $x^{(i)} := R_{U_i}(x)$ and $x^{(0)} := x$. 

\begin{lemma}\label{crmuni} For any $x \in \RR^n$,
\begin{enumerate}
    \item[(i)] $\text{C}_{\text{P}}(x)=P_{W_x}(s)$ satisfies properties (P1) and (P2), for any $s \in S$.
    \item[(ii)] $\text{C}_{\text{P}}(x)$ is unique and therefore
    $\text{C}_{\text{P}}(x) = P_{W_x}(s)$.
\end{enumerate}
\end{lemma}

\begin{proof}
\begin{enumerate}
    \item[(i)] It is evident that $P_{W_x}(s)$ has the property (P1). Moreover, for any $\displaystyle s \in S$ and $1 \le i \le m$, we have $s \in U_i$, whence $\|x - s\| = \|x^{(i)} - s\|$ by Proposition \ref{propprojaff}-(iii).
    Now using Proposition \ref{propprojaff}-(ii) twice, we have for every $i=1,\ldots,m$, that
    $$\|x - P_{W_x}(s)\|^2 = \|x - s\|^2 - \|s - P_{W_x}(s)\|^2 = \|x^{(i)} - s\|^2 - \|s - P_{W_x}(s)\|^2 = \|x^{(i)} - P_{W_x}(s)\|^2,$$
    yielding $\|x - P_{W_x}(s)\| = \|x^{(i)} - P_{W_x}(s)\|$, which  shows that $P_{W_x}(s)$
    satisfies (P2).
    \item[(ii)] If $w, w' \in W_x$ are such that $\|w-x^{(i)}\| = \|w-x\|$ and $\|w'-x^{(i)}\| = \|w'-x\|$ for all $i=1,\ldots,m$, using the identity
    \begin{equation}\label{identiparal}
    \|a-b\|^2-\|b-c\|^2+\|c-d\|^2-\|d-a\|^2 = 2 \langle a-c, d-b \rangle
    \end{equation}
    we get 
    $$0=\|w-x\|^2-\|x-w'\|^2+\|w'-x^{(i)}\|^2-\|x^{(i)}-w\|^2 = 2 \langle w-w',x^{(i)}-x \rangle,$$
    which means that $w-w' \in \text{span}\{x^{(0)}-x,x^{(1)}-x,\ldots,x^{(m)}-x\}^{\perp}$, but 
    $$w, w' \in \text{aff}\{x^{(0)},x^{(1)},\ldots,x^{(m)}\} \implies w-w' \in \text{span}\{x^{(0)}-x,x^{(1)}-x,\ldots,x^{(m)}-x\},$$
    hence $w-w'=0.$ 
\end{enumerate}
\end{proof}

Let us state another important property of the points in the affine space $W_x.$

\begin{lemma}\label{lemwx} For any $w \in W_x$, we have:
$$
\|\text{C}_{\text{P}}(x)-s^\star\|^2 = \|w-s^\star\|^2 - \|w-\text{C}_{\text{P}}(x)\|^2.
$$
\end{lemma}
\begin{proof} 
Lemma \ref{crmuni}-(i) shows that $\text{C}_{\text{P}}(x) = P_{W_x}(s^\star)$ since $s^{\star} \in S$. Now the result follows immediately from Proposition \ref{propprojaff}-(ii) applied to $W_x$ and $s^\star.$
\end{proof}
\vspace*{0.5cm}

The following lemma demonstrates a property that is typical of algorithms designed to solve best approximation problems: applying multiple compositions of the \pCRM{} operator \( \text{C}_{\text{P}}(\cdot) \) does not alter the projection onto the intersection \(S\). Specifically, for any \( x \in \RR^n \) and \( k \in \mathbb{N} \), we have
\[
P_{S}(\text{C}_{\text{P}}^k(x)) := P_{S}(\underbrace{\text{C}_{\text{P}}(\cdots \text{C}_{\text{P}}(\text{C}_{\text{P}}}_{k \text{ times}}(x)) \cdots)) = P_S(x).
\]

\begin{lemma}\label{induccrm} For any $k \ge 0$, one has $P_S(\text{C}_{\text{P}}^k(x)) = s^\star$.
\end{lemma}

\begin{proof}
 We prove the result by induction on $k$. The result trivially holds for $k=0$. Now assuming that the result holds for $k$, since $$\text{C}_{\text{P}}^{k+1}(x) \in W_{\text{C}_{\text{P}}^k(x)},
 $$
 by Lemma \ref{lemwxb} we get $P_S(\text{C}_{\text{P}}^{k+1}(x))=P_S(\text{C}_{\text{P}}^k(x))=s^{\star}$.
\end{proof}

The next result provides two key identities that reveal how the \pCRM{} operator reduces distance to the solution. The first part expresses a Pythagorean-type relation, quantifying the contraction of the distance to \( s^\star \) in terms of the step taken by the method. The second part shows that each application of \pCRM{} simultaneously reduces the distance to the individual sets \( U_i \), reinforcing the interpretation of the method as making globally coherent progress toward feasibility. These relations will be fundamental in proving linear convergence of the \pCRM{} sequence. 

\begin{lemma}\label{lemdistuni} For any $x \in \RR^n,$
\begin{enumerate}
    \item[(i)] $\displaystyle \|\text{C}_{\text{P}}(x)-s^\star\|^2 = \|x-s^\star\|^2 - \|x-\text{C}_{\text{P}}(x)\|^2$.
    \item[(ii)] $\emph{dist}(\text{C}_{\text{P}}(x), U_i)^2 + \emph{dist}(x, U_i)^2 \le \|x-\text{C}_{\text{P}}(x)\|^2$, for all $i=1,\ldots,m$.
\end{enumerate}
\end{lemma}

\begin{proof} \begin{enumerate}
    \item[(i)] This item directly follows from Lemma \ref{lemwx} since $x \in W_x$.
    \item[(ii)] Using Proposition \ref{propprojaff}-(iii), for any $i$ we have $\|P_{U_i}(x)-x\|=\|P_{U_i}(x)-x^{(i)}\|$, while by construction $\|\text{C}_{\text{P}}(x)-x\|=\|\text{C}_{\text{P}}(x)-x^{(i)}\|.$ 
    
    Using  \eqref{identiparal} again, we obtain
    \begin{align}
    0 &= \|P_{U_i}(x)-x\|^2-\|x-\text{C}_{\text{P}}(x)\|^2+\|\text{C}_{\text{P}}(x)-x^{(i)}\|^2-\|x^{(i)}-P_{U_i}(x)\|^2 \nonumber
    \\ &=2\langle P_{U_i}(x)-\text{C}_{\text{P}}(x), x^{(i)}-x \rangle \nonumber
    \\ 
    &=2\langle P_{U_i}(x)-\text{C}_{\text{P}}(x), R_{U_i}(x)-x \rangle \nonumber\\
    &=4\langle P_{U_i}(x)-\text{C}_{\text{P}}(x), P_{U_i}(x)-x \rangle. \label{ortpcp}
    \end{align}
    Then, by Pythagoras theorem, using the fact, by \eqref{ortpcp}, that $P_{U_i}(x)-\text{C}_{\text{P}}(x)$ and $P_{U_i}(x)-x$ are orthogonal, we obtain
    \begin{equation}\label{idenpcp}
        \|P_{U_i}(x)-\text{C}_{\text{P}}(x)\|^2 + \|x-P_{U_i}(x)\|^2 = \|x-\text{C}_{\text{P}}(x)\|^2. 
    \end{equation}
    Finally, since $P_{U_i}(x) \in U_i,$ 
    $$\text{dist}(\text{C}_{\text{P}}(x), U_i) \le \|P_{U_i}(x)-\text{C}_{\text{P}}(x)\|,$$ 
    which, together with \eqref{idenpcp}, implies the result. 
\end{enumerate}
\end{proof}

We now establish a key inequality that serves as the foundation for the linear convergence of \pCRM. The following lemma shows that the circumcenter operator \( \text{C}_{\text{P}} \) acts as a contraction with respect to the distance to the solution \( s^\star \). That is, each iteration strictly reduces the distance to \( s^\star \) by a fixed factor \( r_P < 1 \), independent of the starting point. This result is the critical step that enables us to derive a global linear convergence rate for the method.

\begin{lemma} \label{cvcspm} For some $r_P <1,$ it holds that
\begin{equation}\label{cpstr}
\|\text{C}_{\text{P}}(x) - s^\star\| \le r_P\: \|x - s^\star\|.
\end{equation}

\end{lemma}

\begin{proof} By Lemma \eqref{lemdistuni}-(i),(ii) and Proposition \ref{inducdist}, we have
\begin{align*}
    \|\text{C}_{\text{P}}(x)-s^\star\|^2 &= \|x-s^\star\|^2 - \|x-\text{C}_{\text{P}}(x)\|^2
    \\ &\le \|x-s^\star\|^2 - \max(\text{dist}(x, U_i)^2,i=1,\ldots,m)
    \\ &\le \|x-s^\star\|^2 - \frac{1}{r_m^2} \|x-s^\star\|^2,
\end{align*}
where $r_m>1$.
Therefore \eqref{cpstr} holds with
$$r_P = \sqrt{1-\frac{1}{r_m^2}}<1.$$  
\end{proof}

\subsection{Convergence of \pCRM}

The following theorem establishes the linear convergence of the \pCRM{} method. Building on the contraction property derived earlier, it guarantees that the sequence \( (\text{C}_{\text{P}}^k(x))_k \) converges to the projection of the initial point onto the intersection set \( S \), with a linear rate governed by the geometry of the subspaces involved. This result confirms the efficiency and robustness of \pCRM{} in solving affine feasibility problems.

\begin{theorem}\label{cvcspmt} Let \( x \in \RR^n \) be given. Then, the \pCRM{} sequence \( (\text{C}_{\text{P}}^k(x))_k \) converges linearly to \( s^{\star} \), the projection of the initial point $x$ onto $S$.
\end{theorem}

\begin{proof} By Lemma \ref{cvcspm}, we conclude that, for any $k \ge 1,$
\begin{align*}
    \|\text{C}_{\text{P}}^k(x) - s^\star\| \le r_P\|\text{C}_{\text{P}}^{k-1}(x) - s^\star\|,
\end{align*}
implying
$$\|\text{C}_{\text{P}}^k(x) - s^\star\| \le r_P^k\|x - s^\star\|$$
for $r_P<1$ and the result follows. 
\end{proof}

\vspace*{0.5cm}

We end up this section establishing a relationship between the rates of convergence of the methods F-SPM and \pCRM: \pCRM{} is at least as good as F-SPM.
Relation \eqref{spmcrm} below shows a stronger result: the rate of convergence of \pCRM{} is at least as good as the best possible convergence rate 
$r_{\text{F-SPM}}^{\star}$
of
F-SPM which is given by \eqref{bestcimm} in the space case of Cimmino's method.

\begin{proposition} We have $r_P \le r_{\text{F-SPM}}.$
\end{proposition}

\begin{proof} 
Due to Lemma \ref{lemwx} and the fact that $T_{\text{F-SPM}}(x) \in W_x$ (by Lemma \ref{proptsm}-(i)), we have 
$$\|\text{C}_{\text{P}}(x)-s^\star\|^2 = \|\text{T}_{\text{F-SPM}}(x) - s^\star\|^2 - \|\text{T}_{\text{F-SPM}}(x)-\text{C}_{\text{P}}(x)\|^2,$$ 
which implies
\begin{equation}\label{spmcrm}
\|\text{C}_{\text{P}}(x)-s^\star\| \le \|\text{T}_{\text{F-SPM}}(x)-s^\star\|,
\end{equation} 
which means that every iteration of \pCRM{} is closer to the common limit point $s^{\star}$ than that of F-SPM. 
\end{proof}

\section{Numerical experiments}\label{sec:numsim}

In this section, we present numerical experiments to illustrate the performance of the \pCRM{} method in comparison with \CRM~\cite{crm2}. As far as we know, the experimental (and theoretical) superiority of \CRM{} over other projection-based methods is already well established; see, for instance, \cite{crm_inspired3,crm_inspired2,Araujo:2022, Arefidamghani:2021, Behling:2021a, Behling:2024, Behling:2024b,Behling:2024c}. 

We benchmark both solvers on synthetic data. To this end, we first construct random Gaussian matrices of the form  
\[
(1 - c)A_{[m,n]} + c \cdot \ones_{[m,n]},
\]
where $A_{[m,n]} \sim \mathcal{N}(0,1)$ is a matrix of order $m \times n$ with entries drawn from the standard normal distribution, and $\ones_{[m,n]}$ is a matrix of ones of the same dimensions. The parameter $c \in [0,1]$ controls the coherence among the rows of the matrix; the larger the value of $c$, the greater the coherence.


In our experiments, we consider the case $m > n$, and evaluate three values of $c$: \{\num{0.}, \num{0.1}, \num{0.2}\}. Specifically, we set $m \in \{\num{5000}, \num{7500}, \num{10000}, \num{12500}\}$ and $n \in \{\num{100}, \num{250}, \num{500}\}$. For each configuration, we randomly select a vector $w \in \RR^m$ from the standard normal distribution and define the unique solution to the system of equations as $x^\star \coloneqq A^T w$. The right-hand side of the system is then computed as $b \coloneqq Ax^\star$. 

Next, we construct $\ell \coloneqq \left \lfloor \frac{m}{n}\right \rfloor + 1$ affine subspaces of the form
\[
U_i = \{x \in \RR^n \mid A_i x = b_i\}, \quad i = 1,\ldots,\ell,
\]
where each $A_i$ is the $i$-th block of the matrix $A$, and $b_i$ is the corresponding block of the vector $b$. The matrix $A$ is partitioned into $\ell$ blocks of size at least $(m/\ell) \times n$.

In all experiments, the initial point is set to the null vector, i.e., $x^0 = 0$. The stopping criterion is the relative error, defined as
\[
\texttt{rel\_err} = \frac{\|x^k - x^\star\|}{\|x^\star\|}.
\]
We set the maximum number of iterations to \num{e4} and the tolerance to \num{e-5}.

The computational experiments were performed on a machine equipped with an AMD Ryzen 9 7950X3D 16-Core Processor (32 threads), 128~GB of RAM, and running Ubuntu 24.04. All experiments were implemented in Julia v1.10~\cite{Bezanson:2017} and can be reproduced using the repository \url{https://github.com/lrsantos11/CRM-CFP}. To enhance the performance of \pCRM{}, we implemented a parallel operator for providing reflections using \texttt{ThreadsX.jl}\footnote{Freely available at \url{https://github.com/tkf/ThreadsX.jl}}, a Julia package that enables parallel computation via multithreading. In contrast, the original \CRM{} algorithm performs reflections sequentially in a single thread. This setup allows for a fair comparison that highlights the benefits of the parallel implementation. For completeness, we also run \pCRM{} in single-threaded mode to isolate the impact of parallelization. \CRM{} was also run in multi-threaded mode as well, but we do not report the results since they are similar  to those obtained in single-threaded mode.

Table~\ref{tab:timing_projections_comparison} reports the average wall-clock time (in seconds) and total number of projections for both \pCRM{} and \CRM{} across a range of problem configurations. Each pair of rows corresponds to a fixed problem size, defined by the number of affine subspaces (blocks), number of rows ($m$), and number of columns ($n$). For each configuration, the reported values are averages computed over three coherence levels: $c \in \{0.0, 0.1, 0.2\}$. The fastest runtime and fewest number of projections in each configuration are highlighted in bold.

\begin{table}[htbp]
    \centering
    \caption{Comparison between \pCRM{} and \CRM{} in terms of wall-clock time (in seconds) and total number of projections. The fastest time and fewest projections in each configuration are highlighted in bold.}
    \begin{tabular}{ccccccc}
    \toprule
    Method & \# Blocks & Rows ($m$) & Cols ($n$) & Projections & \multicolumn{1}{c}{Time (s)} \\
    \midrule
    \pCRM{} & 11  & 5000  & 500 & 80.7 & {0.06215} \\
    \CRM{}  & 11  & 5000  & 500 & \textbf{66.0} & \textbf{{0.05167}} \\
    \midrule
    \pCRM{} & 16  & 7500  & 500 & \textbf{96.0} & \textbf{{0.07608}} \\
    \CRM{}  & 16  & 7500  & 500 & \textbf{96.0} & {0.07713} \\
    \midrule
    \pCRM{} & 21  & 5000  & 250 & 126.0 & {0.03423} \\
    \CRM{}  & 21  & 5000  & 250 & \textbf{105.0} & \textbf{{0.02684}} \\
    \midrule
    \pCRM{} & 21  & 10000 & 500 & 126.0 & {0.10172} \\
    \CRM{}  & 21  & 10000 & 500 & \textbf{105.0} & \textbf{{0.08637}} \\
    \midrule
    \pCRM{} & 26  & 12500 & 500 & \textbf{130.0} & \textbf{{0.10346}} \\
    \CRM{}  & 26  & 12500 & 500 & \textbf{130.0} & {0.10597} \\
    \midrule
    \pCRM{} & 31  & 7500  & 250 & \textbf{155.0} & \textbf{{0.04402}} \\
    \CRM{}  & 31  & 7500  & 250 & \textbf{155.0} & {0.04458} \\
    \midrule
    \pCRM{} & 41  & 10000 & 250 & 205.0 & {0.05942} \\
    \CRM{}  & 41  & 10000 & 250 & \textbf{164.0} & \textbf{{0.04845}} \\
    \midrule
    \pCRM{} & 51  & 5000  & 100 & \textbf{153.0} & \textbf{{0.01333}} \\
    \CRM{}  & 51  & 5000  & 100 & \textbf{153.0} & {0.01340} \\
    \midrule
    \pCRM{} & 51  & 12500 & 250 & \textbf{204.0} & \textbf{{0.05638}} \\
    \CRM{}  & 51  & 12500 & 250 & \textbf{204.0} & {0.05656} \\
    \midrule
    \pCRM{} & 76  & 7500  & 100 & \textbf{228.0} & {0.02030} \\
    \CRM{}  & 76  & 7500  & 100 & \textbf{228.0} & \textbf{{0.01871}} \\
    \midrule
    \pCRM{} & 101 & 10000 & 100 & \textbf{202.0} & \textbf{{0.02017}} \\
    \CRM{}  & 101 & 10000 & 100 & \textbf{202.0} & {0.02032} \\
    \midrule
    \pCRM{} & 126 & 12500 & 100 & \textbf{252.0} & {0.02036} \\
    \CRM{}  & 126 & 12500 & 100 & \textbf{252.0} & \textbf{{0.02034}} \\
    \bottomrule
    \end{tabular}
    \label{tab:timing_projections_comparison}
    \end{table}

    This set of experiments serves as a proof-of-concept to illustrate the potential benefits of parallelizing the CRM framework. While \CRM{} remains competitive in several scenarios, \pCRM{} performs comparably or better in others, even on a moderately threaded machine. Since the current implementation relies on multithreading, additional performance gains are expected when executed on systems with a higher number of cores. Moreover, the parallel nature of \pCRM{} makes it a promising candidate for GPU acceleration, where massively parallel architectures can be exploited to further improve performance. These results indicate that \pCRM{} offers a scalable and flexible alternative, particularly suited for large-scale problems and modern computing environments.

\section{Concluding Remarks}

We introduced and analyzed a parallel variant of the Circumcentered Reflection method (\pCRM) for solving best approximation problems over the intersection of affine subspaces. By leveraging the geometric properties of circumcenters and the linear structure of the underlying sets, we established linear convergence of the proposed method under standard regularity conditions. In particular, the contraction property and invariance of projections played central roles in the analysis.

The proposed framework generalizes classical projection-based methods by incorporating symmetry and compositional flexibility, making it particularly well-suited for parallel and distributed implementations. Our numerical experiments confirm that \pCRM{} is a competitive and scalable alternative to its sequential counterpart. Even on a moderately threaded CPU, \pCRM{} achieves performance gains in several configurations, and its potential for speedup increases on systems with more cores. Moreover, the method’s inherent parallelism opens the door to efficient GPU implementations, which could significantly accelerate computations for large-scale problems. The theoretical and practical results for \pCRM{} complement those obtained for F-SPM, highlighting a common geometric foundation underlying both algorithms and establishing \pCRM{} as viable acceleration for simultaneous projection methods.

\bibliographystyle{plain}
\bibliography{references}

\end{document}